\documentclass[oneside,english]{amsart}

\usepackage{amssymb}
\usepackage{amscd}

\numberwithin{equation}{section} 
\numberwithin{figure}{section} 
\theoremstyle{plain}
\newtheorem*{thm*}{Theorem}
\theoremstyle{plain}
\newtheorem{thm}{Theorem}[section]
\theoremstyle{definition}
\newtheorem{defn}[thm]{Definition}
\theoremstyle{plain}
\newtheorem{lem}[thm]{Lemma}
\theoremstyle{plain}

\theoremstyle{plain}

\theoremstyle{remark}

\theoremstyle{remark}
\newtheorem*{acknowledgement*}{Acknowledgement}

\begin{document}

\title[A class of Finsler metrics admitting first integrals]{A class
  of Finsler metrics admitting first integrals}

\author[Bucataru]{Ioan Bucataru}
\address{Faculty of  Mathematics \\ Alexandru Ioan Cuza University \\ Ia\c si, 
  Romania}
\email{bucataru@uaic.ro}
\urladdr{http://www.math.uaic.ro/\textasciitilde{}bucataru/}
\author[Constantinescu]{Oana Constantinescu}
\address{Faculty of  Mathematics \\ Alexandru Ioan Cuza University \\ Ia\c si, 
  Romania}
\email{oanacon@uaic.ro}
\author[Cre\c tu]{Georgeta Cre\c tu}
\address{Departments of  Mathematics \\ Gheorghe Asachi Technical University \\ Ia\c si, 
  Romania}
\email{cretuggeorgeta@gmail.com}

\date{\today}

\begin{abstract}
We use two non-Riemannian curvature tensors, the $\chi$-curvature and the
mean Berwald curvature to characterise a class of Finsler metrics
admitting first integrals. 
\end{abstract}

\subjclass[2000]{53C60, 53B40, 58E30, 49N45}

\keywords{Finsler metric, $\chi$-curvature, scalar mean Berwald curvature,
  first integral}

\maketitle

\section{Introduction}

Finsler geometry is a natural extension of Riemannian geometry and,
while many geometric structures can be extended from the Riemannian to
the Finslerian setting, within the Finslerian context there are many
non-Riemannian geometric quantities, \cite[Ch. 6]{Shen01}.

Existence of first integrals is of great importance, they
provide a lot of information about the corresponding geometry,
including some rigidity results, \cite{CGR20}, \cite{FR16}, \cite{TM03}

In Riemannian geometry, Topalov and Matveev obtained in \cite{TM03},
for two projectively equivalent metrics on an $n$-dimensional manifold, a
set of $n$ first integrals. An extension of this result to the
Finslerian context has been proposed by Sarlet in \cite{Sarlet07}. In
\cite{FR16}, Foulon and Ruggiero have shown the existence of a first
integral for $k$-basic (of isotropic curvature) Finsler surfaces.

It has been proven by Li and Shen in \cite{LS18}, that Finsler metrics of isotropic curvature
can be characterised using the $\chi$-curvature tensor. The
$\chi$-curvature has been introduced by Shen in \cite{Shen13}, in
terms of another important non-Riemannian quantity, the Shen-function ($S$-function) \cite[\S 5.2]{Shen01}.
Since then, a lot of effort has been devoted to study the $\chi$-curvature, \cite{LS15, Mo09, Shen20}. 

In this work we extend the result of Foulon and Ruggiero from \cite{FR16} to Finsler
manifolds of arbitrary dimension, by providing a class of Finsler metrics that
admit first integrals. This class of Finsler metrics is characterised
using the $\chi$-curvature tensor and the mean Berwald curvature,
$E_{jk}=\frac{1}{2}B^i_{ijk}$, where $B^i_{jkl}$ is the Berwald
curvature, \cite[\S 6.1]{Shen01}. Very important in our work is
the fact that the mean Berwald curvature can be expressed also in
terms of the $S$-function. The $S$-function is a Finsler function if
and only if the mean Berwald curvature has maximal possible rank, $n-1$. For a
Finsler function $F$, we denote by $\det g$, the
determinant of its metric tensor $g_{ij}=\frac{1}{2}\frac{\partial^2 F^2}{\partial
  y^i\partial y^j}$, where $y^i$ are the fiber coordinates in the tangent
bundle $TM$.

The main result of this paper provides a class of Finsler metrics that
admit a first integral.

\begin{thm} \label{main} Consider $F$ a Finsler metric that satisfies the
    following two conditions:
\begin{itemize}
\item[i)] the $\chi$-curvature vanishes;
\item[ii)] the mean Berwald curvature has rank $n-1$.
\end{itemize}
Then,
\begin{eqnarray}
\lambda =\frac{-1}{\det g}  \begin{vmatrix}
2FE_{ij} & \displaystyle\frac{\partial
  F}{\partial y^i} \vspace{2mm} \\
\displaystyle\frac{\partial F}{\partial y^j} & 0 
\end{vmatrix} \label{isf}
\end{eqnarray}
 is a first integral for the geodesic spray $G$ of the Finsler metric $F$, which means that $G(\lambda)=0$.
\end{thm}

For a Finsler surface, the first condition of Theorem \ref{main}, $\chi=0$, is equivalent to the
fact that the Finsler metric has isotropic curvature (it is a $k$-basic Finsler metric). Also, in
dimension $2$, the mean Berwald
curvature is proportional to the vertical Hessian of the Finsler metric, the proportionality
factor, the function $\lambda$, was known since Berwald, \cite[(8.7)]{Berwald41}. Hence for Finsler surfaces, the second condition of
Theorem \ref{main} is automatically satisfied. Moreover, the first
integral $\lambda$, given by formula \eqref{isf}, reduces in the
$2$-dimensional case to the first integral $f$ obtained by Foulon and
Ruggiero in  \cite[Theorem B]{FR16}. 

For the proof of Theorem \ref{main}, the two conditions $\chi=0$ and $\operatorname{rank}(E_{ij})=n-1$ tell us that
the $S$-function is a Finsler metric, projectively related to
$F$. Then, we will obtain the first integral $\lambda$, given by \eqref{isf}, using the Painlev\'e first integral, associated to the two
projectively related Finsler metrics $F$ and $S$. 

Next theorem deals with a concrete class of Finsler metrics that satisfy
the second assumption of Theorem \ref{main}.
We say that a Finsler metric $F$ has \emph{scalar mean Berwald curvature} $f$ if the mean Berwald curvature is
proportional to the vertical Hessian of $F$, $2E_{ij}=fF_{y^iy^j}$.  

\begin{thm} \label{main2} Consider $F$ a Finsler metric that satisfies the
    following two conditions:
\begin{itemize}
\item[i)] the $\chi$-curvature vanishes;
\item[ii)] the Finsler metric has scalar mean Berwald curvature $f$. 
\end{itemize}
Then, the scalar mean Berwald curvature satisfies: 
\begin{itemize}
\item[1)] $f$ is a first integral of the Finsler metric $F$.
\item[2)] If $\dim M>2$ then the first integral $f$ is constant. 
\item[3)] If $M$ is compact and $\dim M>2$ then the first integral $f$
  vanishes identically.
\end{itemize}
\end{thm}
The proof of Theorem \ref{main2} is a direct extension, to the
$n$-dimensional case, of the techniques used by Foulon and
Ruggiero in \cite{FR16} to prove the existence of a first integral for $k$-basic
Finsler surfaces. These techniques allow to provide more information
about the first integral and one can further use \cite{CGR20} to obtain a rigidity
result for the class of Finsler metrics with vanishing
$\chi$-curvature and scalar mean Berwald curvature.    

\section{Finsler metrics: a geometric setting and some non-Riemannian
  quantities}

In this work, we assume that $M$ is an $n$-dimensional $C^{\infty}$-
manifold, of dimension $n>1$. We consider $TM$ its tangent bundle and
$T_0M=TM\setminus\{0\}$ the tangent bundle with the zero section
removed. Local coordinates on $M$, denoted by $(x^i)$, induce canonical
coordinates on $TM$ (and $T_0M$), denoted by $(x^i, y^i)$.  On $TM$
there are two canonical structures that we will  use: the Liouville
(dilation) vector field, ${\mathcal  C}=y^i\frac{\partial}{\partial
  y^i}$, and the tangent structure  (vertical endomorphism),
$J=\frac{\partial}{\partial y^i}\otimes dx^i$.   

\subsection{A geometric setting for Finsler metrics}

We will use the Fr\"olicker-Nijenhuis theory to describe the geometric
setting we follow in this work. For a vector-valued $l$-form $L$ on $T_0M$, we denote
by $i_L$ the induced $i_{\ast}$-derivation of degree $(l-1)$ and by
$d_L:=[i_L, d]$ the $d_{\ast}$ derivation of degree $l$, \cite{BD09, Grifone72, GM00, SLK14,
  Youssef94}. For two vector valued forms, an $l$-form $L$ and a
$k$-form $K$, consider the $(k+l)$-form $[K,L]$, uniquely determined
by 
\begin{eqnarray*}
d_{[K,L]}=d_Kd_L-(-1)^{kl}d_Ld_K.
\end{eqnarray*}

A \emph{spray} is a second order vector field $G\in {\mathfrak X}(T_0M)$ such
that $JG={\mathcal C}$ and $[{\mathcal C}, G]=G$. Locally, a spray can
be expressed as 
\begin{eqnarray*}
G=y^i\frac{\partial}{\partial x^i}-2G^i\frac{\partial}{\partial y^i}, 
\end{eqnarray*}
with the functions $G^i(x,y)$ positively $2$-homogeneous in $y$ ($2^+$-homogeneous). A
\emph{geodesic} of a spray $G$ is a smooth curve $c$ on $M$
whose velocity $\dot{c}$ is an integral curve of $G$, $G(\dot{c}(t))=\ddot{c}(t)$.

For a given spray $G$, an orientation preserving reparameterization $t\to
\tilde{t}(t)$, of its geodesics, leads to a new spray
$\widetilde{G}=G-2P{\mathcal C}$, where $P$ is a $1^+$-homogeneous
function on $T_0M$. We say that the  two sprays $G$ and
$\widetilde{G}$ are \emph{projectively related}, while $P$ is called the
\emph{projective factor}.    

\begin{defn} \label{fs} A \emph{Finsler structure} on $M$ is a  continuous function
$F: TM\to [0, +\infty)$ that satisfies the following conditions:
\begin{itemize}
\item[i)] $F$ is smooth on $T_0M$;
\item[ii)] $F$ is $1^+$-homogenous, $F(x,\lambda y)=\lambda F(x,y)$, $\forall
    \lambda>0$, $\forall (x,y)\in T_0M$;
\item[iii)] the metric tensor
\begin{eqnarray*}
g_{ij}=\frac{1}{2}\frac{\partial^2 F^2}{\partial y^i \partial y^j}
\end{eqnarray*} is non-degenerate on $T_0M$.
\end{itemize}
\end{defn}

A Finsler manifold is a pair $(M,F)$, with $F$ a Finsler structure on
the manifold $M$. For a Finsler manifold, one can identify the sphere
bundle $SM$ with the indicatrix bundle $IM=\{(x,y)\in TM,
F(x,y)=1\}$. Geometric objects on $T_0M$ that are invariant under
positive rescaling ($0^+$-homogeneous)
can be restricted to the sphere bundle $SM$.

For a Finsler structure $F$, the metric tensor
$g_{ij}$ can be expressed in terms of the angular metric $h_{ij}$ as follows:
\begin{eqnarray*}
g_{ij}=h_{ij}+\frac{1}{F^2}y_iy_j=h_{ij} + \frac{\partial F}{\partial
  y^i} \frac{\partial F}{\partial y^j}, \quad 
h_{ij}=F\frac{\partial^2 F}{\partial y^i\partial y^j}=FF_{y^iy^j},  
\end{eqnarray*}
where $y_i=g_{ik}y^k=F\frac{\partial F}{\partial y^i}.$ The
regularity condition iii) from Definition \ref{fs} is equivalent to
the fact that the angular metric $h_{ij}$ has rank $n-1$, \cite[Proposition
16.2]{Matsumoto86}. 

For a spray $G$ and a function $L$ on $T_0M$, we consider the
\emph{Euler-Lagrange} $1$-form
\begin{eqnarray}
\delta_GL:={\mathcal L}_Gd_JL-dL=\left\{G\left(\frac{\partial
  L}{\partial y^i}\right) -\frac{\partial L}{\partial x^i}\right\}
  dx^i. \label{el_1}
\end{eqnarray}
Every Finsler metric uniquely determines a \emph{geodesics spray},
solution of the Euler-Lagrange equation $\delta_GF^2=0$.

We recall now the geometric structures induced by a Finsler metric,
and its geodesic spray $G$. We first have the \emph{canonical nonlinear
connection}, characterised by  a horizontal and a vertical projector
on $T_0M$
\begin{equation*}
h=\frac{1}{2}(\operatorname{Id}-[G,J]), \quad v=\frac{1}{2}(\operatorname{Id}+[G,J]).
\end{equation*}
In induced local charts on $T_0M$, the two projectors can be expressed
as:
\begin{eqnarray*}
h=\frac{\delta}{\delta x^i}\otimes dx^i, \ 
  v=\frac{\partial}{\partial y^i}\otimes \delta y^i,  \ \textrm{ where
  } \ \frac{\delta}{\delta x^i}=\frac{\partial}{\partial x^i} - N_i^j
  \frac{\partial}{\partial y^j}, \ \delta y^i=dy^i + N^i_j dx^j \
  \textrm{ and } N^i_j=\frac{\partial G^i}{\partial y^j}.  
\end{eqnarray*}

\begin{lem} \label{2h}
Consider $F$ a Finsler metric and $\widetilde{F}$ a $1^+$-homogeneous
function, nowhere vanishing on $T_0M$. Then, we can
express the determinant of the metric tensor $g_{ij}$
as follows:
\begin{eqnarray}
\det g=-\frac{F^{n+1}}{\widetilde{F}^2} \begin{vmatrix}
\displaystyle\frac{\partial^2 F}{\partial y^i\partial y^j} & \displaystyle\frac{\partial
  \widetilde{F}}{\partial y^i} \vspace{2mm} \\
\displaystyle\frac{\partial \widetilde{F}}{\partial y^j} & 0 
\end{vmatrix} . \label{gg}
\end{eqnarray}   
\end{lem}
\begin{proof}
First, we recall a formula that connects the determinant of the metric tensor
$g_{ij}$ in terms of the angular metric $h_{ij}$, \cite[(1.26)]{Rund59}:
\begin{eqnarray}
\det g=-F^{n-1}
\begin{vmatrix}
\displaystyle\frac{\partial^2 F}{\partial y^i\partial y^j} & \displaystyle\frac{\partial
  F}{\partial y^i} \vspace{2mm} \\
\displaystyle\frac{\partial F}{\partial y^j} & 0 
\end{vmatrix}. \label{gt}
\end{eqnarray}
For the metric tensor $g_{ij}$, consider $\{h_1=\frac{G}{F}$, $h_2, ...,
h_n\}$ an orthonormal horizontal basis and $\{h^i, i=\overline{1,n}\}$, the dual frame. Since,
for $\alpha \geq 2$, $h^{\alpha}(h_1)=0$, and  
\begin{eqnarray*}
h^{\alpha}=h^{\alpha}_i dx^i, \quad
  h_{1}=\frac{y^i}{F}\frac{\delta}{\delta x^i},
\end{eqnarray*}
we obtain $h^{\alpha}_iy^i=0$. Using also $h_{ij}y^j=0$, we obtain,
for each $\alpha\geq 2$, 
that on $T_0M$, 
\begin{eqnarray*}
\begin{pmatrix}
h_{ij} & h^{\alpha}_i \\
h^{\alpha}_j & 0
\end{pmatrix}
\begin{pmatrix}
y^1 \\
\vdots  \\
y^n \\
0
\end{pmatrix}=0
\end{eqnarray*}
and consequently, 
\begin{eqnarray}
\begin{vmatrix}
h_{ij} & h^{\alpha}_i \\
h^{\alpha}_j & 0
\end{vmatrix}=0. \label{halpha} 
\end{eqnarray}
The semi-basic $1$-form $d_J\widetilde{F}$ can be expressed as follows
\begin{eqnarray*}
d_J\widetilde{F} & = & \frac{\partial \widetilde{F}}{\partial y^i}dx^i =
  d_J\widetilde{F}(h_i)  h^i =
  \frac{\widetilde{F}}{F} h^1 + \sum_{\alpha=2}^n
  d_J\widetilde{F}(h_{\alpha})  h^{\alpha} \\ 
& = & \left\{
  \frac{\widetilde{F}}{F}\frac{\partial F}{\partial y^i} +
  \sum^n_{\alpha= 2}J(h_{\alpha})(\widetilde{F}) h^{\alpha}_i
  \right\} dx^i. 
\end{eqnarray*}
In the determinant from \eqref{gt}, we replace $\frac{\widetilde{F}}{F}\frac{\partial
F}{\partial y^i}$ and obtain 
\begin{eqnarray*}
& & \begin{vmatrix}
\displaystyle\frac{\partial^2 F}{\partial y^i\partial y^j}   & \dfrac{\partial
  F}{\partial y^i} \vspace{2mm}  \\
\dfrac{\partial F}{\partial y^j} & 0 
\end{vmatrix} = \left(\frac{F}{\widetilde{F}}\right)^2
\begin{vmatrix}
\displaystyle\frac{\partial^2 F}{\partial y^i\partial y^j}  &
\dfrac{\widetilde{F}}{F} \dfrac{\partial
  F}{\partial y^i} \vspace{2mm}  \\
\dfrac{\widetilde{F}}{F}  \dfrac{\partial F}{\partial y^j} & 0 
\end{vmatrix}  = \left(\frac{F}{\widetilde{F}}\right)^2 \left(
 \begin{vmatrix}
\displaystyle\frac{\partial^2 F}{\partial y^i\partial y^j}  & \dfrac{\partial
  \widetilde{F}}{\partial y^i} \vspace{2mm}  \\
 \dfrac{\partial \widetilde{F}}{\partial y^j} & 0 
\end{vmatrix} \right.\\
& - & \sum_{\alpha=2}^n \left.
\begin{vmatrix}
\displaystyle\frac{\partial^2 F}{\partial y^i\partial y^j}   &
J(h_{\alpha})(\widetilde{F})h^{\alpha}_i \vspace{2mm}  \\
J(h_{\alpha})(\widetilde{F})h^{\alpha}_i  & 0 
\end{vmatrix}\right)  
 \stackrel{\eqref{halpha}}{=}  \left(\frac{F}{\widetilde{F}}\right)^2
\begin{vmatrix}
\displaystyle\frac{\partial^2 F}{\partial y^i\partial y^j}  &
\dfrac{\partial \widetilde{F}}{\partial y^i} \vspace{2mm}  \\
\dfrac{\partial \widetilde{F}}{\partial y^j} & 0 
\end{vmatrix}. 
\end{eqnarray*}
We replace this in formula \eqref{gt} to obtain \eqref{gg} and complete the proof. 
\end{proof}

For a Finsler metric $F$, the regularity condition iii) of Definition
\ref{fs} can be reformulated in terms of the Hilbert $1$-form
$d_JF$. Since $d_JF$ is $0^+$-homogeneous, we can view it as a $1$-form
on $SM$. Due to the fact that the Hilbert $2$-form can be expressed as 
\begin{eqnarray}
dd_JF=\frac{\partial^2F }{\partial y^i\partial y^j}\delta  y^j\wedge
  dx^i = \frac{1}{F}h_{ij}\delta  y^j\wedge dx^i, \label{ddjf}
\end{eqnarray}
it follows that $d_JF$ is a contact structure on $SM$ and hence the
$(2n-1)$- form $\omega_{SM}=d_JF\wedge dd_JF^{(n-1)}$ is a volume form on $SM$.

\subsection{Non-Riemannian structures in the Finslerian setting}

The first non-Riemannian structures, associated to a Finsler metric $F$, are the
\emph{Cartan torsion} and the \emph{mean Cartan torsion}, 
\begin{eqnarray*}
C_{ijk}=\frac{1}{4}\frac{\partial^3 F^2}{\partial y^i\partial
  y^j\partial y^k}=\frac{1}{2}\frac{\partial g_{ij}}{\partial y^k},
  \quad I_k=g^{ij}C_{ijk}.
\end{eqnarray*}
A Finsler metric reduces to a Riemannian metric if and only if the
(mean) Cartan torsion vanishes.

The \emph{curvature} of the nonlinear connection determined by the
geodesic spray $G$ is defined by
\begin{eqnarray*}\label{R-curv-tensor}
R=\frac{1}{2}[h,h]=\frac{1}{2}R^i_{jk}\frac{\partial}{\partial y^i}\otimes
dx^j\wedge dx^k = \frac{1}{2}\left(\frac{\delta N^i_j}{\delta x^k} - \frac{\delta
    N^i_k}{\delta x^j} \right) \frac{\partial}{\partial y^i}\otimes
dx^j\wedge dx^k. 
\end{eqnarray*}

The canonical nonlinear connection provides a tensor derivation
on $T_0M$, the \emph{dynamical
  covariant derivative} $\nabla$, whose action on functions and vector
fields is given by \cite[(21)]{BD09}:
\begin{eqnarray*}
\nabla(f)=G(f), \forall f\in C^{\infty}(TM), \quad \nabla
  X=h[G,hX]+v[G,vX], \forall X \in {\mathfrak X}(TM).
\end{eqnarray*}

The geodesic spray $G$ of a Finsler metric induces a linear connection
on $T_0M$, the Berwald connection, \cite[\S 8.1.1]{SLK14}, with two curvature components, the \emph{Berwald
curvature} and the \emph{Riemannian curvature}, \cite[\S 6.1,\S, 8.1]{Shen01}:
\begin{eqnarray*}
B^i_{jkl}=\frac{\partial^3 G^i}{\partial y^j \partial y^k \partial y^l},
  \quad R^i_{jkl}= \frac{\partial R^i_{kl}}{\partial y^j}.
\end{eqnarray*}
The \emph{mean Berwald curvature} of a spray $G$ is defined as
\cite[Def. 6.1.2]{Shen01}
\begin{eqnarray}
E_{jk}=\frac{1}{2}B^i_{ijk} = \frac{1}{2} \frac{\partial^3
  G^i}{\partial y^i \partial y^j \partial y^k}. \label{ejk}
\end{eqnarray}
\begin{defn} \label{smbc}
A Finsler metric has \emph{scalar mean Berwald curvature} if the mean
Berwald curvature is proportional to the angular metric, which means
that there exists a $0^{+}$-homogeneous function $f$ on $T_0M$ such
that
\begin{eqnarray}
E_{ij}=\frac{1}{2}\frac{f}{F}h_{ij}=\frac{f}{2}\frac{\partial^2 F}{\partial
  y^i\partial y^j}. \label{ef}
\end{eqnarray}
\end{defn}
In \cite{CS05}, Chen and Shen study Finsler metrics of \emph{isotropic mean
Berwald curvature}, with a similar definition as above, with
$f$ being a scalar function on $M$.  

In the next lemma we prove that in dimension $n>2$, Finsler
metrics of \emph{scalar mean Berwald curvature} have isotropic mean
Berwald curvature. In other words, the scalar mean Berwald curvature $f$ is
constant along the fibres of $T_0M$.
\begin{lem} \label{fconst}
Consider $F$ a Finsler metric of \emph{scalar mean Berwald curvature}
$f$. If $n>2$, then $f$ is constant along the fibres of $T_0M$.
\end{lem}
\begin{proof}
From the definition of the mean Berwald curvature \eqref{ejk} we obtain that 
its vertical derivative is a $(0,3)$-type tensor symmetric in all
three arguments, therefore for a Finsler
metric of \emph{scalar mean Berwald curvature} we have 
\begin{eqnarray*}
\frac{\partial E_{ij}}{\partial y^k}= \frac{\partial E_{ik}}{\partial
  y^j}  & \stackrel{\eqref{ef}}{\Longrightarrow} & \frac{\partial f}{\partial y^k}
  \frac{\partial^2 F}{\partial y^i\partial y^j} =  \frac{\partial f}{\partial y^j}
  \frac{\partial^2 F}{\partial y^i\partial y^k} \Longrightarrow \frac{\partial f}{\partial y^k}
  h_{ij}=  \frac{\partial f}{\partial y^j}  h_{ik} \\ &
                                                        \Longrightarrow & \frac{\partial f}{\partial y^k}
  \left(g_{ij} - \frac{1}{F^2}y_iy_j\right) =  \frac{\partial f}{\partial y^j}
  \left(g_{ik} - \frac{1}{F^2}y_iy_k\right).
\end{eqnarray*}
In the last formula above, we multiply with $g^{il}$, the inverse of
the metric tensor and obtain:
\begin{eqnarray*}
\frac{\partial f}{\partial y^k}
  \left(\delta^l_j - \frac{1}{F^2}y^ly_j\right) =  \frac{\partial f}{\partial y^j}
  \left(\delta^l_k - \frac{1}{F^2}y^ly_k\right).
\end{eqnarray*}
Now, if we use that $f$ is $0^+$-homogeneous and take the trace,
$j=l$, we obtain $(n-2){\partial f}/{\partial y^k}=0$. Since $n>2$, we
obtain that the function $f$ is constant along the fibres of $T_0M$. 
\end{proof}

Due to the $2^{+}$-homogeneity of the spray coefficients $G^i$, it
follows that $E_{ij}y^j=0$, hence $\operatorname{rank}(E_{ij})\leq n-1$. In the $2$-dimensional case, we obtain that
the mean Berwald curvature has rank $1$, it is proportional to the
angular metric $h_{ij}$ (of rank $1$ as well), and hence all $2$-dimensional Finsler manifolds
have scalar mean Berwald curvature. The proportionality factor has been known
since Berwald, \cite[(8.7)]{Berwald41}, but it has been shown only
recently that it is a first integral for $k$-basic Finsler surfaces,
\cite[Theorem B]{FR16}.

The Berwald connection is not a metric connection, with respect to the
metric tensor of a Finsler structure. Due to this non-metricity property of the Berwald
connection, it follows that the $(0,4)$-type Riemann curvature tensor
$R_{ijkl}=g_{is}R^s_{jkl}$ is not skew-symmetric in the first two
indices, \cite[(10.6)]{Shen01}, and hence $R^i_{ikl}\neq 0$. A measure of this failure is given by
the $\chi$-curvature, \cite[Lemma 3.1]{Shen20}:
\begin{eqnarray*}
\chi_j=-\frac{1}{2}R^i_{ijk}y^k.
\end{eqnarray*}
This non-Riemannian quantity has been introduced by Shen in
\cite{Shen13}. 

The key ingredients we will use in this work are the $\chi$-curvature,
the mean Berwald curvature, and the fact that both curvature tensors can be expressed
in terms of yet another non-Riemannian quantity, the $S$-function. 

For a fixed vertically invariant volume form $\sigma(x) dx^1\wedge dx^2 \wedge \cdots \wedge
dx^n \wedge dy^1\wedge dy^2 \wedge \cdots \wedge
dy^n$ on $TM$, \cite[p. 490]{SLK14}, we consider the Shen-function ($S$-\emph{function}) and the
\emph{distortion}  $\tau$, \cite[\S 5.2]{Shen01},
\begin{eqnarray}
S=G\left(\tau\right), \quad \tau=\frac{1}{2}\ln{\frac{\det g}{\sigma}}. \label{scurvature}
\end{eqnarray}  

From the various expressions of the $\chi$-curvature, we will use its
expression in terms of the $S$-function, \cite[(1.10)]{Shen13},
\begin{eqnarray}
\chi=\frac{1}{2}\delta_GS=\frac{1}{2}\left\{\nabla\left(\frac{\partial
  S}{\partial y^i}\right)-\frac{\delta S}{\delta x^i}\right\}dx^i. \label{chi}
\end{eqnarray}

The mean Berwald curvature can also  be expressed in terms of the
$S$-function as follows, \cite[(6.13)]{Shen01}:
\begin{eqnarray}
E_{ij}=\frac{1}{2}\frac{\partial^2 S}{\partial y^i\partial
  y^j}. \label{es}
\end{eqnarray}
In view of formula \eqref{es}, the second assumption of Theorem
\ref{main} or \ref{main2}, assures that the vertical Hessian of the
$S$-function has maximal rank $(n-1)$. Therefore, we can interpret the
$S$-function as a Finsler metric on its own.

\section{Proof of Theorem \ref{main}}

For the proof of Theorem \ref{main}, we proceed with the
following steps. We show first that the two assumptions of Theorem
\ref{main} assure that the $S$-function is a Finsler metric,
projectively related to $F$. Then, we obtain the first integral
\eqref{isf} using the Painlev\'e first
integral associated to the two projectively related Finsler metrics
$S$ and $F$. 

Two Finsler metrics $F$ and $\widetilde{F}$ are projectively related if
their geodesic sprays $G$ and $\widetilde{G}$ are projectively
related. One can characterise projective equivalence of two Finsler
metrics $F$ and $\widetilde{F}$ using the following equivalent forms of
Rapcs\'ak equations, \cite[\S 9.2.3]{SLK14}:
\begin{itemize}
\item[($R_1$)] $\delta_G\widetilde{F}=0$; 
\item[($R_2$)] $d_hd_J\widetilde{F}=0$.
\end{itemize}

In Riemannian geometry, Topalov and Matveev \cite[Theorem 1]{TM03}
associate to each pair of geodesically equivalent metrics a set of
$n$ first integrals. An extension of this result, to the Finslerian
setting, has been proposed by Sarlet in \cite{Sarlet07} and his Ph.D
student Vermeire \cite{Vermeire06}. 

In the next lemma, we show that two projectively related Finsler
metrics $F$ and $\widetilde{F}$ induce a first integral (Painlev\'e first
integral). This first integral, given by formula \eqref{io}, is the
Finslerian extension of the first integral determined by two projectively
equivalent Riemannian metrics, \cite[Theorem 2]{Matveev06}.
\begin{lem} \label{painleve}
Consider $F$ and $\widetilde{F}$, two projectively related Finsler
metrics. Then, 
\begin{eqnarray}
I_0=\frac{\widetilde{F}}{F}\left(\frac{\det g}{\det \widetilde{g}}\right)^{\frac{1}{n+1}} \label{io}
\end{eqnarray}
is a first integral for $F$.
\end{lem} 
\begin{proof}
For a Finsler metric $F$, the dynamical covariant derivative of its
metric tensor vanishes, \cite{B07}, hence:
\begin{eqnarray*}
\nabla(g_{ij})=G(g_{ij})-g_{im}N_j^m-g_{mj}N_i^m=0.
\end{eqnarray*} 
Contracting with $g^{ij}$, we obtain
\begin{eqnarray*}
g^{ij}G(g_{ij})=g^{ij}(g_{im}N_j^m+g_{mj}N_i^j)=2N_i^i, \quad
  \textrm{and \ hence \ } 
  N_i^i=\frac{1}{2}G(\ln (\det g) ).
\end{eqnarray*}
The two Finsler metrics $F$ and $\widetilde{F}$ being projectively
related, their geodesic sprays and nonlinear connections are connected through 
\begin{equation*}
\widetilde{G}=G-2P\mathcal{C}, \quad \widetilde{G^i}=G^i+Py^i, \quad \widetilde{N^i_j}=N^i_j+\frac{\partial P}{\partial y^j}y^i+P\delta^i_j.
\end{equation*}
If in the last formula above we take the trace $i=j$, it follows that the projective factor $P$ is given by 
\begin{eqnarray*}
P = \frac{1}{n+1}(\widetilde{N}^i_i-N_i^i) =
  \frac{1}{2(n+1)}G\left(\ln \left( \frac{\det \widetilde{g}}{\det
  g}\right)\right). \label{p1} 
\end{eqnarray*}
We also use the alternative expression of the projective factor $P$,
\begin{eqnarray*}
P=\frac{G(\widetilde{F})}{2\widetilde{F}}=\frac{1}{2}G\left(\ln \widetilde{F}\right).
\end{eqnarray*}
By comparing the two expressions of the projective factor $P$, we
obtain $G(I_0)=0$, which concludes the proof of our lemma.
\end{proof}

We will give the proof of Theorem \ref{main} now. The second
assumption ii) on Theorem \ref{main} together with formula \eqref{es}
assure that the angular metric of the $S$-function has rank $n-1$ and
therefore $S$ is a Finsler metric. The vanishing of the
$\chi$-curvature \eqref{chi} assures that the Finsler metric $S$ is
projectively related to $F$. In view of Lemma \ref{painleve} we obtain
that 
\begin{eqnarray*}
I_0=\frac{S}{F}\left(\frac{\det g}{\det s}\right)^{\frac{1}{n+1}}
\end{eqnarray*}
is a first integral for the Finsler metric $F$.

We will use Lemma \ref{2h} for the Finsler metric $S$ and the
$1^+$-homogenous function $F$. According to formula \eqref{gg}, we can
express the determinant of the metric tensor $s_{ij}$ as follows:
\begin{eqnarray*}
\det s=-\frac{S^{n+1}}{F^2} \begin{vmatrix}
\displaystyle\frac{\partial^2 S}{\partial y^i\partial y^j}  &
\dfrac{\partial F}{\partial y^i} \vspace{2mm}  \\
\dfrac{\partial F}{\partial y^j} & 0 
\end{vmatrix} \stackrel{\eqref{es}}{=} -\frac{S^{n+1}}{F^2} \begin{vmatrix}
2E_{ij}  &
\dfrac{\partial F}{\partial y^i} \vspace{2mm}  \\
\dfrac{\partial F}{\partial y^j} & 0 
\end{vmatrix} = -\frac{S^{n+1}}{F^{n+1}} \begin{vmatrix}
2FE_{ij}  &
\dfrac{\partial F}{\partial y^i} \vspace{2mm}  \\
\dfrac{\partial F}{\partial y^j} & 0 
\end{vmatrix}. . 
\end{eqnarray*}
Since $I_0$ is a first integral for the Finsler metric $F$, it follows that 
\begin{eqnarray*}
\frac{1}{I_0^{n+1}} =
  \frac{ F^{n+1}}{ S^{n+1}} \frac{\det s}{\det
  g} =  \frac{-1}{\det g}\begin{vmatrix}
2FE_{ij}  &
\dfrac{\partial F}{\partial y^i} \vspace{2mm}  \\
\dfrac{\partial F}{\partial y^j} & 0 
\end{vmatrix}
\end{eqnarray*}
is also a first integral for $F$ that coincides with $\lambda$ given by formula \eqref{isf}.

The first two assumptions of Theorem \ref{main} tell us that $S$ is a
Finsler metric projectively related to $F$. One can use this and \cite[Theorem 1]{TM03} and
\cite[Theorem 2]{Sarlet07} to provide a set of $n$ first integrals for
Finsler metric with vanishing $\chi$-curvature and mean Berwald
curvature of maximal rank. 

\section{Proof of Theorem \ref{main2}}

\subsection{Partial proof of Theorem \ref{main2}}

First we prove the first two conclusions of Theorem \ref{main2}, using
Theorem \ref{main}. For this proof it is essential that the scalar mean
Berwald curvature $f$ is nowhere vanishing, hence we cannot reach the
third conclusion of Theorem \ref{main2} using these techniques. 

In view of the equivalence of the two Rapcs\'ak equations $R_1$ and $R_2$, we can
 reformulate the vanishing of the $\chi$-curvature \eqref{chi} as 
$d_hd_JS=0$. Using also the assumption that $F$ has scalar mean
Berwald curvature, we obtain that the Hilbert $2$-form of the $S$-function can be
written as follows
\begin{eqnarray}
\nonumber dd_JS & = & d_vd_JS=\dfrac{\partial^2S}{\partial y^i\partial y^j}\delta
  y^i\wedge d x^j = 2E_{ij}\delta y^i\wedge
  dx^j\\ 
& = & f \dfrac{\partial^2F}{\partial y^i\partial y^j}\delta
  y^i\wedge d x^j = fdd_JF. \label{ddjs}
\end{eqnarray}

For a non-vanishing scalar mean Berwald curvature $f$, it follows from
\eqref{ddjs} that $\operatorname{rank}\left(\dfrac{\partial^2S}{\partial y^i\partial
    y^j}\right)=n-1$ and hence $S$ is a Finsler metric.

We will express now, the first integral $\lambda$,
\eqref{isf}, using the assumption that $F$ has scalar mean Berwald
curvature $f$. 
We have 
\begin{eqnarray}
\lambda =\frac{-1}{\det g}  \begin{vmatrix}
2FE_{ij} & \displaystyle\frac{\partial
  F}{\partial y^i} \vspace{2mm} \\
\displaystyle\frac{\partial F}{\partial y^j} & 0 
\end{vmatrix} = \frac{-F^{n-1}}{\det g}  \begin{vmatrix}
f \dfrac{\partial^2 F}{\partial y^i \partial y^j} & \displaystyle\frac{\partial
  F}{\partial y^i} \vspace{2mm} \\
\displaystyle\frac{\partial F}{\partial y^j} & 0 
\end{vmatrix} \stackrel{\eqref{gt}}{=} f^{n-1}.
\end{eqnarray}
Since $\lambda$ is a first integral, it follows that $f$ is also a
first integral for $F$, and this proves the first conclusion of Theorem \ref{main2}.

According to Lemma \ref{fconst}, the scalar mean Berwald curvature
$f$ is a scalar function on $M$, which means that $d_Jf=0$. We use now that $G(f)=0$, which can
be written as $\nabla f=0$. If we apply $d_J$ to this formula and use the commutation
rule for $\nabla$ and $d_J$, \cite[(2.11)]{BC15}, we obtain
\begin{eqnarray*}
0=d_J\nabla f=\nabla d_J f + d_hf.
\end{eqnarray*}  
Therefore, $d_hf=0$ and hence $f$ is a constant, which proves the
second conclusion of Theorem \ref{main2}.

\subsection{Complete proof of Theorem \ref{main2}}

In this section we present a proof of Theorem \ref{main2}, independent
of the results of Theorem \ref{main}, by
extending to the $n$-dimensional case, the techniques of
\cite{FR16}. This method allows to provide more information about the
first integral, when the base manifold is compact. 
 
The mean Cartan torsion can be expressed in terms of the
distortion $\tau$, and it does not depend on the fixed volume form on
$M$, 
\begin{eqnarray*}
I_k=\frac{1}{2}g^{ij}\frac{\partial g_{ij}}{\partial
  y^k}=\frac{\partial}{\partial y^k}(\ln\sqrt{\det g})=\frac{\partial
  \tau}{\partial y^k}, \ 
 I=I_kdx^k=d_J(\ln\sqrt{\det g}) = d_J\tau.
\end{eqnarray*}

The key ingredient in this proof is the following $1$-form
\begin{eqnarray}
\label{alpha1} \alpha & = & i_{[J,G]}{\mathcal L}_G I = \nabla I_k dx^k - I_k \delta
  y^k =\nabla d_J\tau - d_v\tau \\
& = & d_J\nabla \tau - d_h\tau - d_v\tau =
  d_J\nabla \tau - d\tau = d_JS - d\tau. \nonumber
\end{eqnarray}
In the $2$-dimensional case, this form reduces to the form $\alpha$ from \cite[\S 2]{FR16}.

We will use the last expression from \eqref{alpha1} of the form $\alpha$ to connect it
with the $\chi$-curvature:
\begin{eqnarray}
{\mathcal L}_G\alpha = {\mathcal L}_Gd_JS - {\mathcal L}_Gd\tau =
  {\mathcal L}_Gd_JS - dS=\delta_GS=2\chi. \label{lgalpha} 
\end{eqnarray}
In view of this formula, the $\chi$-curvature vanishes if and only if
the form $\alpha$ is invariant by the geodesic flow. Moreover, the
$\chi$-curvature vanishes if and only if the $S$-function satisfies the Rapcs\'ak equation $\delta_GS=0$, which is
equivalent to $d_hd_JS=0$.  

Therefore, we can express the $2$-form $d\alpha$ as follows:
\begin{eqnarray*}
d\alpha = dd_JS=d_hd_JS+d_vd_JS=\frac{\partial^2S}{\partial y^i\partial
  y^j} \delta y^i\wedge dx^j= 2E_{ij} \delta y^i\wedge dx^j.
\end{eqnarray*}

If we consider now the assumption that the Finsler metric has scalar
mean Berwald curvature, then the $2$-form $d\alpha$ is proportional to
the Hilbert $2$-form $dd_JF$:
\begin{eqnarray}
d\alpha=2E_{ij} \delta y^i\wedge dx^j = \frac{f}{F}h_{ij}\delta
  y^i\wedge dx^j= f dd_JF. \label{dalpha}
\end{eqnarray}

From formula \eqref{lgalpha} we obtain that $\chi=0$ implies
${\mathcal L}_G\alpha=0$ and therefore ${\mathcal L}_Gd\alpha=0$. In
view of formula \eqref{dalpha} and using the fact that  ${\mathcal
  L}_Gdd_JF=0$ we obtain $G(f)=0$, which means that the scalar
mean Berwald curvature $f$ is a first integral for the geodesic flow
$G$.

Using Lemma \ref{fconst} we obtain that the scalar
mean Berwald curvature $f$ is a scalar function on $M$, hence
$df=d_hf$. From formula \eqref{dalpha}, we obtain 
\begin{eqnarray*}
0=d^2\alpha = df\wedge dd_JF = d_hf\wedge d_vd_JF=\frac{1}{2}\left(\frac{\partial
  f}{\partial x^i} h_{kj} - \frac{\partial f}{\partial x^j} h_{ki}
  \right) dx^i\wedge dx^j\wedge \delta y^k. 
\end{eqnarray*}
It follows that
\begin{eqnarray*}
\frac{\partial f}{\partial x^i} h_{kj} = \frac{\partial f}{\partial
  x^j} h_{ki} \Longrightarrow \frac{\partial f}{\partial x^i} \left(
  g_{kj} - \frac{1}{F^2} y_ky_j\right)  = \frac{\partial f}{\partial
  x^j} \left( g_{ki} - \frac{1}{F^2} y_ky_i\right). 
\end{eqnarray*}
In the last formula above, we multiply with $g^{il}$ and obtain:
\begin{eqnarray*}
\frac{\partial f}{\partial x^i} \left(
  \delta^l_j - \frac{1}{F^2} y^ly_j\right)  = \frac{\partial f}{\partial
  x^j} \left( \delta^l_i - \frac{1}{F^2} y^lky_i\right).
\end{eqnarray*}
If we take the trace $l=j$, we obtain 
\begin{eqnarray*}
(n-2) \frac{\partial f}{\partial x^i} =-\frac{1}{F^2}G(f)y_i.
\end{eqnarray*}
Now using that $G(f)=0$, we obtain that the scalar function $f$ is
constant if $\dim M>2$.

To complete the proof of Theorem \ref{main2}, we need the
following lemma that gives new properties for the first integral $f$ and can be
useful for some rigidity results. 
\begin{lem} \label{intf}
Let $(M, F)$ be a compact Finsler manifold with vanishing
$\chi$-curvature and of scalar mean Berwald curvature $f$. Then,
\begin{eqnarray}
\int_{SM}f\omega_{SM}=0. \label{intfo}
\end{eqnarray}
\end{lem}
\begin{proof}
By Stokes Theorem we have that
\begin{eqnarray*}
0=\int_{SM}d\left(\alpha \wedge d_JF \wedge (dd_JF)^{n-2}\right) =
  \int_{SM} d\alpha \wedge d_JF \wedge (dd_JF)^{n-2} - \int_{SM}\alpha
  \wedge (dd_JF)^{n-1}.
\end{eqnarray*}
We will prove now that on $SM$, $\alpha\wedge (dd_JF)^{n-1} =0$.

Let $\lambda_1, .., \lambda_{n-1}$ be the non-zero eigenvalues of
the angular metric $h_{ij}$, $h_1, ..., h_{n-1}$ the corresponding
horizontal eigenvectors and $v_i=Jh_i$, $i\in \{1,...,n-1\}$. Then, 
$\{h_1,..., h_{n-1}, v_1, ..., v_{n-1}\}$ is a local frame of the $(2n-2)$-dimensional
distribution $\operatorname{Ker}(d_JF)$ on $SM$. We consider also the local co-frame $\{h^1,...,
h^{n-1}, v^1, ..., v^{n-1}\}$. Using the expression \eqref{ddjf} of the
Hilbert $2$-form, $dd_JF$, it follows that 
\begin{eqnarray*}
(dd_JF)^{n-1} =\lambda_1 \cdots \lambda_{n-1} h^1\wedge \cdots
  h^{n-1}\wedge v^1\cdots \wedge v^{n-1}.
\end{eqnarray*}
Since $i_G\alpha=0$, it follows that $\alpha \in \operatorname{span}
\{h_1,..., h_{n-1}, v_1, ..., v_{n-1}\}= \operatorname{Ker}(d_JF)$, we obtain that $\alpha\wedge
(dd_JF)^{n-1} =0$ on $SM$. Now, using \eqref{dalpha}, we obtain 
\begin{eqnarray*}
0=\int_{SM} d\alpha \wedge d_JF \wedge (dd_JF)^{n-2} = \int_{SM} f dd_JF \wedge d_JF \wedge (dd_JF)^{n-2}=\int_{SM}f\omega_{SM}.
\end{eqnarray*}
\end{proof}
If $\dim M>2$ then $f$ is constant and using formula \eqref{intfo} we obtain that
$f=0$, which completes the proof of Theorem \ref{main2}.

Existence of first integrals for Finsler manifolds can be used to
provide rigidity results under some topological restrictions: 
\begin{itemize}
\item compact surface, without conjugate points and of genus greater
  than one, \cite[Theorem A]{FR16}; 
\item compact manifold, without conjugate points and of uniform visibility, for
dimension $n>2$, \cite[Theorem A]{CGR20}. 
\end{itemize}

If $M$ is a compact manifold of dimension $n>2$, with
vanishing $\chi$-curvature and of scalar mean Berwald curvature $f$, we obtain that
$f=0$. Using formula \eqref{dalpha}, it follows that the form
$\alpha$, given by \eqref{alpha1}, is closed. Using the assumptions
of \cite[Theorem A]{CGR20} we can conclude that the form $\alpha$
is exact. Assume $\alpha=dh$, for some function $h$ on $T_0M$. Since
$i_G\alpha=0$, it follows that $G(h)=0$ and $h$ is a first integral
for the geodesic flow. Using again \cite[Theorem A]{CGR20} we obtain
that the function $h$ is constant, then $\alpha=0$. The expression
\eqref{alpha1} of the form $\alpha$ allows to conclude that the mean Cartan tensor vanishes,
$I=0$, and hence $(M,F)$ is a Riemannian manifold.

\subsection*{Acknowledgments} We express our thanks to J\'ozsef Szilasi for his comments and
suggestions on this work.


\begin{thebibliography}{99}  



\bibitem{Berwald41} {Berwald, L.}: \emph{On Finsler and Cartan
  geometries. III. Two-dimensional Finsler spaces with rectilinear
  extremals}, Ann. of Math., \textbf{42} (1) (1941), 84--112.

\bibitem{B07} {Bucataru, I.}: \emph{Metric nonlinear connections},
  Differential Geom. Appl., \textbf{25}(3) (2007), 335--343.

\bibitem{BD09} {Bucataru, I., Dahl, M.}: \emph{Semi-basic 1-forms and
    Helmholtz conditions for the inverse problem of the calculus of
    variations},  J. Geom. Mech., \textbf{1}(2) (2009), 159–-180.

\bibitem{BC15} {Bucataru, I., Constantinescu, O.}: \emph{Generalized
    Helmholtz conditions for non-conservative Lagrangian systems},
  Math. Phys. Anal. Geom., \textbf{18} (1) (2015), Art. 25, 24 pp.



\bibitem{CS05} {Chen, X., Shen, Z.}: \emph{On Douglas metrics},
  Publ. Math. Debrecen, \textbf{66}(3-4) (2005), 503--512.

\bibitem{CGR20} {Chimentona, A.G., Gomes, J.B., Ruggiero, R.O.}:
  \emph{Gromov-hyperbolicity and transitivity of geodesic flows in
    n-dimensional Finsler manifolds}, Differential Geom. Appl., \textbf{68} (2020), 101588.

\bibitem{FR16} {Foulon, P., Ruggiero, R.O.}: \emph{A first integral for
    $C^{\infty}$, $k$-basic Finsler surfaces and applications to
    rigidity}, Proc. Amer. Math. Soc.,\textbf{144} (9) (2016),
  3847--3858. 

\bibitem{Grifone72}  {Grifone, J.}: \emph{Structure presque-tangente et
    connexions I}, Ann. Inst. Fourier, \textbf{22} (1972), 287--334.

\bibitem{GM00}{Grifone, J., Muzsnay, Z.}: \emph{Variational Principles
    For Second-Order Differential Equations}, World Scientific, 2000.

\bibitem{LS15} {Li, B., Shen, Z.}: \emph{Ricci curvature tensor and
    non-Riemannian quantities}, Canad. Math. Bull. \textbf{58}(3)
  (2015), 530--537.  

\bibitem{LS18} {Li, B., Shen, Z.}: \emph{Sprays of isotropic
    curvature}, Int. J. Math., \textbf{29} (1) (2018), 1850003, 12pp. 

\bibitem{Matsumoto86} {Matsumoto, M.}: \emph{Foundations of Finsler geometry
    and special Finsler spaces}, Kaiseisha Press, 1986.

\bibitem{Matveev06} {Matveev, V.}: \emph{Geometric explanation of the
    Beltrami theorem}, Int. J. Geom. Methods Mod. Phys., \textbf{3} (3) (2006), 623--629.

\bibitem{Mo09} {Mo, X.}: \emph{On the non-Riemannian quantity H of a
    Finsler metric}, Differential Geom. Appl., \textbf{27} (1) (2009),
  7--14.

\bibitem{Rund59} {Rund, H.}: \emph{The differential geometry of
    Finsler spaces}, Springer, 1959. 

\bibitem{Sarlet07} {Sarlet, W.}: \emph{A recursive scheme of first
    integrals of the geodesic flow of a Finsler manifold}, SIGMA
  Symmetry Integrability Geom. Methods Appl., \textbf{3} (2007), Paper 024, 9 pp. 

\bibitem{Shen01} {Shen, Z.}: \emph{Differential geometry of spray and
Finsler spaces}, Springer, 2001.

\bibitem{Shen13} {Shen, Z.}: \emph{On some non-Riemannian quantities
    in Finsler geometry}, Canad. Math. Bull., \textbf{56} (1) (2013),  184--193.

\bibitem{Shen20} {Shen, Z.}: \emph{On sprays with vanishing X-curvature}, arXiv:2008.07732. 

\bibitem{SLK14} {Szilasi, J., Lovas, R., Kert\'esz, D.}:
  \emph{Connections, sprays and Finsler structures}, World Scientific,
  2014.

\bibitem{TM03} {Topalov, P., Matveev, V.S.}: \emph{Geodesic
    equivalence via integrability}, Geom. Dedicata, \textbf{96}
  (2003), 91--115. 

\bibitem{Youssef94} {Youssef, N.L.}: \emph{Semi-projective changes},
  Tensor, N.S. \textbf{55}(1994), 131--141. 

\bibitem{Vermeire06} {Vermeire, F.}: \emph{A class of recursion operators
on a tangent bundle}, Ph.D thesis, University of Gent, 2006.

\end{thebibliography}
\end{document}